\documentclass[11pt]{article}

\usepackage{latexsym,verbatim,ifthen,graphicx,mathrsfs}
\usepackage[english]{babel}
\usepackage[utf8]{inputenc} 	
\usepackage{amssymb}
\usepackage{amsmath}
\usepackage{amsthm}	
\usepackage{tikz-cd}
\usepackage{mathtools}  % Para usar \xhookrightarrow[sub]{sup}
\usepackage{microtype}		 % hace mejoras para ahorrar espacio y que quede todo mejor.
\usepackage[overload]{empheq}

\usepackage{calc}
\newlength\myheight
\newlength\mydepth
\settototalheight\myheight{Xygp}
\usetikzlibrary{matrix,arrows,decorations.pathmorphing}
\usepackage[all]{xy}				
\usepackage{hyperref}
\hypersetup{
	colorlinks=true,
	allcolors=blue,
}   		
\usepackage{anysize}
\usepackage{cancel}									
\marginsize{2.5cm}{2.5cm}{2cm}{2cm}
\usepackage{enumerate}

\usepackage{mathrsfs}									
\usepackage{fourier}			
\usepackage[Sonny]{fncychap}	
\usepackage{fancyhdr}

\newtheorem{theorem}{Theorem}[section]
\newtheorem{lemma}[theorem]{Lemma}
\newtheorem{proposition}[theorem]{Proposition}
\newtheorem{corollary}[theorem]{Corollary}
\theoremstyle{definition}
\newtheorem{example}[theorem]{Example}

\theoremstyle{definition}
\newtheorem{remark}[theorem]{Remark}

\newtheorem{definition}[theorem]{Definition}

\newtheorem{thm}{Theorem}

\DeclareMathOperator{\id}{id}
\DeclareMathOperator{\idd}{id}

\DeclareMathOperator{\ms}{\mathbb S}

\renewcommand{\P}{{\mathcal{P}}}

\newcommand{\E}{{\mathcal{E}}}
\newcommand{\D}{\mathcal{\D}}
\newcommand{\C}{\mathcal{\C}}

\newcommand{\Sym}{{\operatorname{Sym}}}

\definecolor{red}{rgb}{1,0.1,0.1}
\definecolor{blue}{rgb}{0.1,0.1,1}
\definecolor{green}{rgb}{0,100,0}

\begin{document}

\title{An obstruction theory for strictly commutative algebras in positive characteristic}

\author{Oisín Flynn-Connolly}

\date{}
\maketitle

\abstract{This is the first in a sequence of articles exploring the relationship between commutative algebras and $E_\infty$-algebras in characteristic $p$ and mixed characteristic. In this article we lay the groundwork by defining a new class of cohomology operations over $\mathbb F_p$ called cotriple products, generalising Massey products. We compute the secondary cohomology operations for a strictly commutative dg-algebra and the obstruction theories these induce, constructing several counterexamples to characteristic 0 behaviour, one of which answers a question of Campos, Petersen, Robert-Nicoud and Wierstra. We construct some families of higher cotriple products and comment on their behaviour. Finally, we distinguish a subclass of cotriple products that we call higher Steenrod operations and conclude with our main theorem, which says that $E_\infty$-algebras can be rectified if and only if the higher Steenrod operations vanish coherently.}

\section{Introduction}
Since its introduction by Quillen \cite{quillen69} and Sullivan \cite{Sullivan77}; rational homotopy theory has probably become the single most successful subfield of algebraic topology. The latter approach to the theory reduces the study of rational topological spaces to that of commutative dg-algebras. 

\medskip

One of the central observations of \cite{Sullivan77} was that it was possible to replace the $E_\infty$-algebra $C^\bullet(X, \mathbb Q)$ with a strictly commutative model $A_{PL}(X)$. In positive characteristic, this is not possible because the Steenrod operations act as obstructions, and, in particular, for spaces we have that the zeroth Steenrod power operation  $P^0$ never vanishes (see Proposition \ref{prop:spaces_are_never_commutative}). The main goal of this article is to investigate the precise relationship between strictly commutative and $E_\infty$-algebras in positive characteristic.

\medskip

It is here that we introduce the key idea of this article; that is, that rectifiablity in characteristic $p$ should be studied in a similar manner to formality in characteristic 0. In characteristic 0, Massey products provide higher obstructions to formality. Massey triple products correspond to relations in the cohomology algebra and higher Massey products correspond to syzygies between those relations. For an $E_\infty$-algebra to be rectifiable, its Steenrod operations must vanish, but we also will need to impose conditions ensuring that syzygies between them should vanish as well. This will be complicated by the fact that strictly commutative dg-algebras do have one Steenrod operation - the Frobenius map - and some higher cohomology operations coming from that.

\medskip

It is worth noting that, despite not modelling spaces directly, commutative algebras retain some significant advantages over $E_\infty$-algebras. The most significant of these is computational. Commutative algebras, due to the fact that they are generated by a single operation with the simplest possible behaviour, of all the algebraic objects appearing in geometry and topology, are the most amenable to computer algebra techniques. In contrast, $E_\infty$-algebras are generated by infinitely many operations, are generally large, unwieldy, and very few operations, with some notable exceptions \cite{medina21}, on them have been, or are likely to be, automated.

\medskip

We develop an obstruction theory, determined by \emph{cotriple products}, a non-linear generalisation of both Massey products and Steenrod operations, designed to handle the extra invariants given by the Frobenius map. More precisely, these operations are defined to be differentials in a spectral sequence associated to the operadic cotriple resolution. We show that they also admit a description that is highly reminscent of defining systems for Massey products, complete with well-defined indeterminacies. In many situations, they also possess the main advantage of Massey products, the indeterminacy can be computed directly from the product set without needing to compute the full spectral sequence (Proposition \ref{productsetease}).

\medskip

We then study cotriple resolutions in the context of the commutative operad in positive characteristic. The main result is that the secondary cotriple products consist of traditional Massey products along with two additional types of operation coming from the Frobenius map. We compute the indeterminacy of these additional operations and use them to construct examples of commutative dg-algebras that are formal over $\mathbb Q$ (Example \ref{ex:algebrasweaklyhomotopicoverQbutnotFp}) but not over $\mathbb F_p$, and algebras with a divided power structure on cohomology that are not weakly equivalent to a divided powers algebra (Example \ref{ex:algebranoteqtodividedpower}).

\medskip

In \cite[Section 0.3]{campos23}, the authors pose the following question: \emph{If two commutative dg algebras are quasi-isomorphic as associative dg algebras, must they be quasi-isomorphic also as commutative dg algebras?} They then settle the question in characteristic 0.

    \begin{theorem}\cite[Theorem A]{campos23}
        Let $A$ and $B$ be two commutative dg algebras over a field of characteristic zero. Then, $A$ and $B$ are quasi-isomorphic as associative dg algebras if and only if they are also quasi-isomorphic as commutative dg algebras. 
    \end{theorem} 
It is perhaps unsurprising that cotriple products offer a useful perspective in characteristic $p$. In particular, we are able to provide an explicit counterexample (Theorem \ref{mycounterexample}) in characteristic 2 which extends to a general method for constructing counterexamples in characteristic $p$ for odd primes. 

Aside from cotriple products, the other key tool in this construction are \textit{cup-1 algebras}, which are a combinatorially convenient way to find and work with algebras that have a commutative cohomology ring, but which are not themselves commutative. We note that these have a rich theory, which still undergoing development, for more details see \cite{porter25}. 
\medskip

We then turn to the study of higher order operations, this turns out to be much more subtle than with Massey produts, as the existence of a higher \emph{homotopy invariant} operation does not necessarily follow automatically from the vanishing of a lower order one (Example \ref{counterexampletohomotopytype}). We define an infinite family of higher operations coming from the iterations of the Frobenius map and compute their indeterminacy.

\medskip 

Finally, we conclude with a necessary and sufficient condition for an $E_\infty$-algebra to be quasi-isomorphic to a commutative algebra. We first define higher Steenrod operations as subset of the cotriple operations. Then we have the following rectification result.
\begin{thm}
    Let $A$ be an $E_\infty$-algebra over $\mathbb F_p$. Then $A$ is rectifiable if and only if its higher Steenrod operations vanish coherently.
\end{thm}
The coherent vanishing condition, which is inspired by the following theorem of Deligne, Griffiths, Morgan and Sullivan.
\begin{theorem}
\cite{deligne75}
    Let $A$ be a commutative dg-algebra over $\mathbb Q$. Let $\mathfrak m = (\mathbb \Sym(\bigoplus_{i=0}^\infty V_i), d)$ be the minimal model for $A$. Then $A$ is formal if and only if, there is in each $V_i$ a complement $B_i$ to the cocycles $Z_i$, $V_i= Z_i \oplus B_i$, such that any closed form, $a$, in the ideal, $I((\bigoplus_{i=0}^\infty B_i)$, is exact. 
\end{theorem}

\medskip

One can also ask whether the homotopy type of a commutative dg-algebra coincides with its homotopy type as an $E_\infty$-algebra. While this is true in characteristic 0, we believe that in positive characteristic the answer is no. This can be shown with the aid of third order cotriple products, and we intend to return to it in separate work.

\medskip

\subsubsection*{Structure of the article}
First we recall some preliminaries on  commutative dg-algebras, $E_\infty$-algebras and Steenrod operations. Then in Section 3, we define cotriple products, our non-linear generalisation of Massey products, and show that they are homotopy invariant for well-behaved classes of algebras. In Section 4, we study the case of strictly commutative dg-algebras and construct various counterexamples. Finally in Section 5, we prove our rectification result.\medskip

\subsubsection*{Acknowledgements}
I would like to thank my PhD advisor Grégory Ginot for numerous useful discussions and providing feedback on this paper. I would also like to thank Geoffroy Horel, Jos\'{e} Moreno-Fern\'{a}ndez and Fernando Muro for useful conversations and comments; and furthermore to the anonymous referee and Birgit Richter for their time, effort, and excellent advice. This project has received funding from the
European Union’s Horizon 2020 research and innovation programme under the Marie Skłodowska-
Curie grant agreement No 945322. \raisebox{-\mydepth}{\fbox{\includegraphics[height=\myheight]{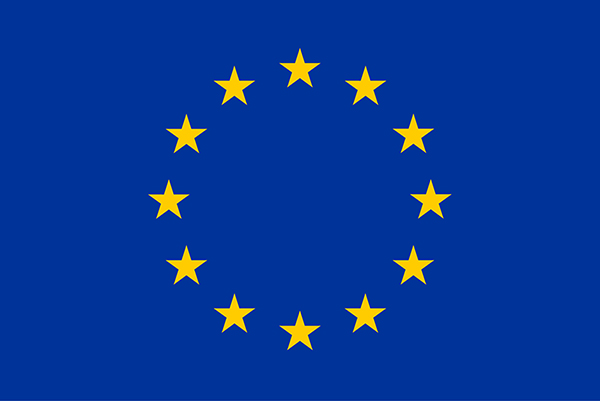}}}

\subsubsection*{Notation and conventions}
In this article, we work on the category of unbounded chain complexes over some base field or ring with cohomological convention.
That is, 
the differential $d:A^*\to A^{*+1}$ of a chain complex $\left(A,d\right)$
is of degree $1$.
The degree of a homogeneous element $x$ is denoted by $|x|$. 
The symmetric group on $n$ elements is denoted $\mathbb{S}_n$.
We follow the Koszul sign rule.
That is, the symmetry isomorphism $U\otimes V \xrightarrow{\cong} V \otimes U$ that identifies two graded vector spaces is given on homogeneous elements by
$u\otimes v\mapsto (-1)^{|u||v|} v\otimes u$.
Algebras over operads are always differential graded (dg) and cohomological.
We will frequently omit the adjective "dg" and assume it is implicitly understood.
Finally, when $A$ is a ring of characteristic $p$ for a prime $p$, $A^p$ will be the subring $\{a^p: a\in A \}$.

This is a short article and we do not intend to load it excessively with recollections; so therefore we refer to \cite{loday12} for the definition of an operad and other basic notions.

\section{Preliminaries}
\subsection{Three flavours of algebra over an operad}
\label{dividedpoweralgebras}
Divided power algebras were first introduced for the commutative operad by Cartan \cite{cartan54}, generalised to the general operadic setting by Fresse \cite{fresse00} and studied further by Ikonicoff \cite{Ikonicoff23}. In this section, 
we carefully define them and, in particular, we explain the free graded commutative divided powers algebra 
on a free module. 

Recall that an algebra $A$ over a operad $\P$ is defined to be an algebra over the following monad 
\[
\mathcal{P}(V) = \bigoplus_{n \geq 0} \left(\mathcal{P}(n) \otimes V^{\otimes n}\right)_{\mathbb S_n},
\]
where the coinvariants are taken with respect to the action of the symmetric group on $V^{\otimes n}$ given by permutation of factors in the tensor product. In characteristic zero, assuming $\P(n)$ is finitely generated as a representation of $\ms_n$, this is the same monad taking invariants
\[
\mathcal{P}(V) = \bigoplus_{n \geq 0} \left(\mathcal{P}(n) \otimes V^{\otimes n}\right)^{\mathbb S_n}.
\]
However, when we are not working over a field of characteristic zero, these notions do not coincide. This motivates the following definition.
\begin{definition}\cite{fresse00}
Let $A$ be dg-module over a commutative unital ring. We say that $A$ is a \emph{$\P$-algebra} if it is an algebra over the monad
    \[
\mathcal{P}(V) = \bigoplus_{n \geq 0} \left(\mathcal{P}(n) \otimes V^{\otimes n}\right)_{\mathbb S_n}.
\]
An algebra over the monad 
\[
\
\Gamma\mathcal{P}(V) = \bigoplus_{n \geq 0} \left(\mathcal{P}(n) \otimes V^{\otimes n}\right)^{\mathbb S_n},
\]
is referred to as a \emph{divided powers-$\P$ algebra.}
\end{definition}
    In general, there is a universal map from coinvariants to invariants, the norm map 
    $$
    \mathcal{P}(V) \to \Gamma\mathcal{P}(V)
    $$
    It follows that every divided powers $\P$-algebra is a $\P$-algebra. The image of the norm map is usually denoted $\Lambda \mathcal{P}(V)$ and is the third flavour of algebras.

\medskip

We shall mainly be interested in the case $\P = \mathsf{Com}$, so it will be useful to have explicit descriptions in this case.
\medskip
\begin{example}
    When $\P = \mathsf{Com}$ and $V = \mathbb F_p$, with the single basis element $x$,
    $$
    \Gamma \mathsf{Com}(V)= \begin{cases}
     \mathbb F_p[x_1, x_2, \dots]/(x_1^p, x_2^p, \dots ) &\mbox{with $|x_k| = k|x|$, when $|x|$ is even.} 
     \\
     \mathbb F_p[x]/(x^2) &\mbox{otherwise.}
    \end{cases}
    $$
\end{example}
\subsection{\texorpdfstring{$E_\infty$}{E-infinity}-algebras and Steenrod operations}
\label{sec:einfintyalge}
As a representation of $\mathbb S_n$, $\mathsf{Com}(n) = \mathbb k$ is not free. Unfortunately this means that the associated Schur functor behaves poorly in positive characteristic (see Example \ref{ex:symbadlybehaved}). The traditional approach to remedying this is to find a weak equivalence of operads $\E \xrightarrow{\sim}\mathsf{Com}$ such that, for each $n$, the action of $\mathbb S_n$ on $\E(n)$ is free. Any such operad is called an $E_\infty$-operad. The model we shall use in this article is the Barratt-Eccles operad; where the action of the symmetric group is free,  which we briefly recall. For more details, see \cite{berger04}.
\begin{definition}
\label{Def:BE}
The simplicial sets defining the Barratt-Eccles operad in each arity are of the form
$$
\E(r)_n = \{(w_0,\dots, w_n)\in \ms_r\times \cdots \times \ms_r \}
$$
equipped with face and degeneracy maps
\begin{align*}
    d_i(w_0, \dots,w_n) = (w_0, \dots, w_{i-1}, \hat{w_i}, w_{i+1}, \dots, w_n)\\
    s_i(w_0, \dots,w_n) = (w_0, \dots, w_{i-1}, w_i, w_i, w_{i+1}, \dots, w_n).
\end{align*}
The group $\ms_r$ acts on $\E(n)$ diagonally, that is to say if $\sigma \in \ms_n$ and $(w_0,\dots, w_n) \in \E(n)$ then
$$
(w_0,\dots, w_n)\ast \sigma = (w_0\ast \sigma,\dots, w_n\ast \sigma)
$$

Finally the compositions are also defined componentwise via the explicit composition law of 
\begin{gather*}
\label{co}
    \gamma: \ms(r)\otimes\ms(n_1)\otimes\cdots\otimes\ms(n_r)\to\ms(n_1+\cdots+n_r)
\\
    (\sigma, \sigma_1, \dots, \sigma_r)\mapsto \sigma_{n_1\cdots n_r} \circ (\sigma_1 \times \cdots\times \sigma_r)
\end{gather*} 
where $\sigma_{n_1\cdots n_r}$ is the permutation that acts on $n_1+\cdots+ n_r$ elements, by dividing them into $r$ blocks, the first of length $n_1$, the second of length $n_2$ and so on. It then rearranges the blocks according to $\sigma$, maintaining the order within each block. 
\end{definition}
\begin{remark}
    The Barratt-Eccles operad as defined above is an operad in simplicial sets. However it becomes an operad concentrated in graded chain complexes, concentrated in non-negatively degrees, after applying the singular chains functor.  When we work in cohomological grading and cochain complexes, it is therefore concentrated in non-positive degrees and is unbounded below. In this article, the notation $\E$ shall always refer to the operad in chain complexes.
\end{remark}
One can check that the free algebra functor $\E(-)$ respects homotopy equivalences. In other words, if $V \xrightarrow{\sim} W$ is a homotopy equivalence of dg-modules; then $\E(V)\xrightarrow{\sim}\E(W)$ is a homotopy equivalence. Furthermore, the cohomology of $\E(V)$ has an additional grading induced by the operadic degree 
$$
\E(V)_i = \bigoplus_{r=1}^\infty\E(r)_i \otimes_{\ms_r} V^{\otimes r}.
$$
There is then an isomorphism $\Sym\left(H^*\left(V\right)\right) \xrightarrow{\sim}
H^*\E(V)_0,$  where $\Sym(-)$ is the symmetric algebra functor.
\begin{definition} \cite{may70}
    Let $V$ be a dg-module over $\mathbb F_p$. The Steenrod algebra on $V$ is the cohomology group $\mathcal A(V) = H^*(\E(V))$. 
\end{definition}
\begin{remark}
    Let $V$ be a non-negatively graded dg-module. Then the Steenrod algebra $\mathcal A(V)$ will not be bounded below. However, if $V$ is a non-positively graded dg-module, $\mathcal A(V)$ will also be concentrated in non-positive degrees.
\end{remark}
Given an $E_\infty$-algebra $A$, one has a map
$$
\mathcal A(H^*(A)) \xrightarrow{H^\ast(\gamma)} H^*(A).
$$
where $\gamma$ is the $E_\infty$-algebra map $\E(A)\xrightarrow{\gamma} A$.
In other words, the cohomology $H^\ast(\gamma)$ of an $E_\infty$-algebra always carries an action of $\mathcal A$.
\subsubsection{Rectification}
There is a weak equivalence of operads $\phi : \mathcal E \xrightarrow{\sim} \mathsf{Com},$ so it is natural to ask whether or not the pair $(\phi^*,\phi_!)$ forms a Quillen equivalence
between $E_\infty$-algebras and Com-algebras. If there is, then \emph{rectification} is said to
occur. With coefficients in $\mathbb Q,$ this is indeed the case; see for example \cite{white17}.
In particular, this implies that that every $E_\infty$-algebra $A$ has a commutative model given by $\phi_!(A).$
\subsubsection{The homotopy theory of $E_\infty$-algebras and commutative dg-algebras}
In this subsection we shall discuss the existence of model structures on categories of $\P$-algebras and specialise to the cases of $E_\infty$-algebras. One has the following general fact.
\begin{theorem}\cite{hinich01} \label{modelcategory}
Let $\P$ be an $\ms$-split (or cofibrant) operad over a commutative ring $R$. Then the category of $\P$-algebras over $R$ is a closed model category with quasi-isomorphisms as the weak equivalences and surjective maps as fibrations.
\end{theorem}
The Barratt-Eccles operad is $\ms$-split. This immediately gives the model structure on $E_\infty$-algebras over $\mathbb F_p.$
\begin{definition}
    The model category $E_\infty \mathsf{-alg}$ of $E_\infty$-algebras is the category of algebras over the operad given by cochains on the Barratt-Eccles operad, in chain complexes over $\mathbb F_p$. It has  quasi-isomorphisms of cochain complexes as weak equivalences and surjective maps as fibrations.
\end{definition}

\medskip

We have already mentioned that in characteristic 0, the homotopy theory of commutative dg-algebras coincides with that of $E_\infty$-algebras.  In positive characteristic the relationship is much more complex.

\medskip

Quasi-isomorphisms of commutative dg-algebras are those algebra maps that are quasi- isomorphisms of chain complexes. The category of commutative algebras thus has a well-defined homotopy category given by localising at these quasi-isomorphisms.

\medskip
\section{Cotriple products}
\subsection{Sullivan algebras}
In this subsection, we explain how to construct a semi-free, and therefore, in the presence of a model category, cofibrant, resolution of an algebra over an arbitrary operad $\P$. The following results are likely well known to experts but we could not locate a proof in the literature. 

\begin{definition}
\label{def:sullivan_algebra}
    Let $\P$ be an operad over a field and let $A$ be a $\P$-algebra. A \emph{Sullivan model (or resolution)} for $A$ is a quasi-free algebra $(\P(\bigoplus_{i = 0}^\infty V_i), d)$ and a map $f: (\P(\bigoplus_{i = 0}^\infty V_i), d) \xrightarrow{\sim} A$ such that
    \begin{itemize}
    \item 
        the map $f|_{V_0}: V_0\to A$ is a is a weak equivalence of dg-vector spaces. 
        \item 
        the differential satisfies $d(V_k) \subseteq \P(\bigoplus_{i = 0}^{k-1} V_i)$ \emph{(the nilpotence condition)}. In particular $V_0 = H^*(A).$
        \item 
           the map $V_k\oplus (\P(\bigoplus_{i = 0}^{k-1} V_i) \to A$ is a weak equivalence for each $k.$
    \end{itemize}
\end{definition}
Every $\P$-algebra $A$  over a field admits a Sullivan model. The argument is essentially the same as \cite[Proposition 12.1]{felix01}.
\begin{proposition}
\label{prop:sullivan_model}
Suppose either that $\P$ be an operad over a field and $A$ is an arbitrary $\P$-algebra or that $\P$ is an operad over a ring and $A$ is  $\P$-algebra equipped with a weak equivalence of dg-modules map $H^*(A) \xrightarrow{\sim} A$. Then $A$ has a Sullivan resolution $m:(\P(V), d) \xrightarrow{\sim} A$. 
\end{proposition}
\begin{proof}
    Let $V_0= H^\ast(A)$ and choose a map 
    $$
    m_0: (\P(V_0), 0) \to A
    $$
    such that $V_0\to A$ is an isomorphism on cohomology.
    
    Suppose that $m_0$ has been extended to $m_k: \left( \P \left( \bigoplus_{i=0}^k V_i \right), d\right) \to A $. Let $z_\alpha$ be cocycles in $\P \left( \bigoplus_{i=0}^\infty V_i \right)$ such that $[z_\alpha]$ is a basis for $\ker H(m_k).$ Let $V_{k+1}$ be a graded space with basis $\{v_\alpha\}$ in $1-1$ correspondence with the $z_\alpha,$ and with $|v_\alpha| = |z_\alpha| - 1.$ Extend $d$ to a derivation in  $\P \left( \bigoplus_{i=0}^k V_i \right)$ by setting $dv_\alpha = z_\alpha.$ Since $d$ has odd degree, $d^2$ is a derivation. Since $d^2v_\alpha = dz_\alpha = 0$, we have that $d^2 = 0.$
    
\medskip

    Since $H(m_k)[z_\alpha] = 0,$ $m_kz_\alpha = da_\alpha, a_\alpha\in A.$ Extend $m_k$ to a graded $\P$-algebra morphism $m_{k+1}:\P \left( \bigoplus_{i=0}^{k+1} V_i \right) \to A $ by setting $m_{k+1}v_\alpha = a_\alpha.$ Then $m_{k+1}dv_\alpha = dm_{k+1} v_\alpha,$ and so we havc $m_{k+1}d = d m_{k+1}.$

    \medskip

    We conclude our construction by setting $V = \bigoplus^\infty_{i=0} V_i.$ We have a map $m : (\P(V), d) \xrightarrow{\sim} A$ such that $m|_{V_k} = m_k.$ Since $m|_{\P(V_0)} = m_0$ and $H(m_0)$ is surjective, $H(m)$ is surjective too. If $H(m)[z] = 0$ then, since $z$ must be in some $\P \left( \bigoplus_{i=0}^k V_i \right),$ one has $H(m_k)[z] = 0$. But then, by construction $z$ is a coboundary in $\P \left( \bigoplus_{i=0}^{k+1} V_i \right),$ and so is a coboundary in $(\P(V), d).$ Therefore $H(m)$ is an isomorphism.

\medskip
    
The nilpotence condition on $d$ is built into the construction.
\end{proof}
\begin{definition}
    We refer to the semifree algebra appearing in the previous proof
    $$
    m_k: \left( \P \left( \bigoplus_{i=0}^k V_i \right), d\right) \to A 
    $$
    as a \emph{step $k$ Sullivan resolution} of $A.$ 
\end{definition}
Sullivan algebras behave like cofibrant objects, in the sense that they have the left lifting property against all  surjective weak equivalences.
\begin{theorem}
\label{thm:llp}
    Let $m:(\P(V), d)$ be a Sullivan algebra. Then the map $0 \to (\P(V), d)$ has the left lifting property against all  surjective weak equivalences $p : C\to D$.
\end{theorem}
\begin{proof}
We have the following diagram.
    $$
    \begin{tikzcd}
        0 \rar{f} \dar{i} & C \arrow[d, "p"]
        \\
        (\P(V), d) \arrow[r, "g"] \arrow[ru, dotted, "h"] & D
    \end{tikzcd}
    $$
   The map $h$, is determined by the image of the generators $v\in V$. We do this inductively. For $v\in V_0$, we define $h_0: \P(V_0) \to C$ by  $f(v) = c$ where $p(c) = g(v),$ which is defined since $p$ is surjective. Now assume that we have defined $\left( \P \left( \bigoplus_{i=0}^k V_i \right), d\right) \to C.$  For $v_\alpha \in V_{k+1}$, we claim it possible to find $c_\alpha \in C$ such that  $ p(c_\alpha) = g(v_\alpha)$ and $d c_\alpha = h(z_\alpha).$ Let $dv_\alpha =z_\alpha$, then $h(z_\alpha)$ is a cocycle. Choose any lift $c'_\alpha$ such that $dc'_\alpha = h(z_\alpha).$ Moreover, we have $p(h(z_\alpha)) = g(z_\alpha) = g(d v_\alpha) = dg(v_\alpha)$. Therefore $g(v_\alpha)-p(c'_\alpha) $ is nullhomologous in $D.$ There is therefore $K_\alpha \in D$ such that $dK_\alpha =g(v_\alpha)-p(c'_\alpha).$ The cochain $K_\alpha$ has a preimage in $L_\alpha\in C$ and we define $c_\alpha = L_\alpha + c'_\alpha.$  Finally, we define $h(v_\alpha) = c_\alpha$.
\end{proof}
\begin{remark}
    In particular, the surjective weak equivalences are the acyclic fibrations in the setting of Theorem \ref{modelcategory}. Therefore, Sullivan models are cofibrant objects in this model category. 
\end{remark}
\begin{remark}
\label{rem:sullivan_algebras_respect_weak_eq}
An observation that will be important later is that, when $g: (\P(V), d) \to D$ is a Sullivan resolution, the map $h$ constructed in the proof also satisfies the conditions of Definition \ref{def:sullivan_algebra}. 
\end{remark}
In what follows, we say that an operad $\P$ \emph{reflects homotopy equivalences} if the free algebra functor $\P(-)$ sends quasi-isomorphisms of cochain complexes to quasi-isomorphisms of $\P$-algebras.
\begin{proposition}
\label{prop:sullivan_algebras_respect_weak_eq}
    Let $\P$ be an operad over a field such that $\P$ reflects homotopy equivalences. Let $A$ be a $\P$-algebra and let $m:(\P(V), d) \xrightarrow{\sim} A$ be a Sullivan resolution. Then, for any $\P$-algebra $B$ weakly equivalent to $A$ there exists a $m'$ such that $m':(\P(V), d) \xrightarrow{\sim} B$ is a Sullivan resolution.
\end{proposition}
\begin{proof}
    The $\P$-algebra $B$ must be connected to $A$ via zig-zags of quasi-isomorphisms. So it suffices to show that Sullivan resolutions can be transferred across quasi-isomorphisms in both directions. So if there is a quasi-isomorphism $f:A\xrightarrow{\sim} B$, then $m' = f\circ m$. Now, suppose there is a quasi-isomorphism $f:B\xrightarrow{\sim} A.$ One can associate an acyclic fibration to $f$ by choosing an acyclic complement $W$, in the category of graded vector spaces,  in $A$ to the image of $f$ ie.\ $A = \operatorname{Im} f \oplus W$ as graded vector spaces. This is not necessarily a chain complex as it may not contain $dw$ for each $w\in W$ but an acyclic chain complex $W'$ may be obtained from it by directly adding such missing coboundaries. Then the map $f':B\oplus \P(W')\to A$ is an acyclic fibration and so is the projection $\pi: B\oplus \P(W') \to B.$ The desired map is then $m' = \pi\circ h $, where $h:(\P(V), d) \xrightarrow{\sim} B\oplus \P(W')$ is the Sullivan resolution coming from applying from Remark \ref{rem:sullivan_algebras_respect_weak_eq} to $f'$.
\end{proof} 
\subsection{Cotriple products in positive characteristic}

In this subsection we introduce a theory of higher Massey-like products for algebras over operads over a field of arbitrary characteristic paralleling that of \cite{massey58, Mur21, flynn23}. The operads here are permitted to have a differential. The underlying idea is similar to that in \cite{flynn23}; we simply wish to define a higher operation for every syzygy. In the setting of Koszul operads over $\mathbb Q$, the existence of minimal models for operads made this straightforward. In this context, Massey products correspond to differentials in the Eilenberg-Moore spectral sequence.

\medskip

In the positive characteristic setting, things are much more complicated. The existence of the Frobenius map produces syzygies that mix the algebra and the operad in ways that were not possible in zero characteristic. We therefore define a non-linear generalisation of a Massey product called a \emph{cotriple product} that captures this phenomenon.

\medskip 

\noindent{\bf{The cotriple resolution. }}
Let $\P$ be an operad and let $A$ be a $\mathcal P$-algebra. The \emph{cotriple resolution} $\mathsf{Res}_{\P}(A)$, which is a model for $A$ in the category of free $\P$-algebras, is a simplicial $\P$-algebra defined as follows.
In simplicial degree $n$, one has
$$
\mathsf{Res}_{\P}(A)_n =
    \P^{\circ (n+1)}(A)
$$
where $\P^{\circ i}$ indicates that the free algebra functor is applied $i$ times. The face maps are given by
$$
d_n^i:  \P^{\circ (n+1)}(A) \to  \P^{\circ n}(A)
$$
$$
d_n^i = \begin{cases}        
    \P(A)^{\circ i}\circ \gamma_{\P}\circ  \P(A)^{\circ (n-i-1)}(\Bbbk) & \mbox{for } n  = 0, \dots, n-1.
    \\
    \P^{\circ (n-1)}\circ \gamma_A & \mbox{for } i = n.  
\end{cases}
$$
where $\gamma_\P$ is the operadic composition map and $\gamma_A: \P(A)\to A$ is the $\P$-algebra map. The degeneracy maps are defined by
$$
s_n^i: \P^{\circ (n+1)}(A) \to \P^{\circ (n+2)}(A)
$$
$$
s^i_n = \P^{\circ i}\circ u \circ \P^{\circ (n-i)}(A)
$$
where $u: \P \to \P \circ \P$ is the unit map. This object can be realized as a chain complex $(|\mathsf{Res}_{\P}(A)|, d +\partial)$ where the $d$ is the internal differential on $A$ and $\partial$ is the differential coming from the simplicial structure we have just defined. The cotriple resolution is not novel to this work, it was developed and studied by Fresse \cite{fresse16}.

\medskip 

\noindent{\bf{The cotriple spectral sequence. }} The cotriple resolution  $(|\mathsf{Res}_{\P}(A)|, d +\partial)$ is defined as the realization of a simplicial object and therefore admits a skeletal filtration. The \emph{cotriple spectral sequence} is the associated spectral sequence. A morphism of augmented $\P$-algebras naturally induces a morphism of the corresponding spectral sequences.
The $E^0$-page of this spectral sequence is explicitly given by
\[
E_0^{p,q} =  \P^{\circ p}(A)^{p+q}
\]
where the $p+q$ grading is the total grading. The differential $d^0$ is therefore the usual differential on $ \P^{\circ p}(A).$

Suppose now that the free algebra functor $\P$ reflects homotopy equivalences and therefore $H^*(\P(A)) = \mathcal B(H^*(A))$ for a functor $\mathcal B;$ it follows that the $E_1$-page of the spectral sequence is
$$
E_1^{p,q} = \left(\mathcal B^{\circ p}(H^*(A))\right)^{p+q}
$$
 and the differential on this page is therefore entirely determined by the depth 2 component of the codifferential.
For $\P = \E,$ $E^2_{p,q}$ is the free Steenrod algebra applied $p$ times  to the cohomology of $A.$ 

\medskip 

\noindent{\bf{Cotriple products. }} We shall refer to the higher differentials in this spectral sequence as \textit{cotriple products}. When $\P$ reflects homotopy equivalences, homotopy invariance is immediate; as any weak equivalence of $\P$-algebras induces an isomorphism of $E_1$-pages. Cotriple products are therefore well-defined as elements of $E_n^{p,q}$ in the spectral sequence.

\medskip
\noindent{\bf{Cotriple products in terms of Sullivan algebras. }}We shall give a second description which is more in line with the classical definition of Massey products in terms of defining systems. Theorem \ref{Thm : Differentials EMSS and cotriple products} gives the proof of the correspondence. 
\begin{definition}
\label{def:product_set}
     Let $A$ be a $\P$-algebra and fix a choice of $f: (\P (\bigoplus_{i=0}^N V_i), d) \xrightarrow{\sim} A$ an $N$-step Sullivan model for $A.$ Consider the ideal $I(\bigoplus_{i=1}^N V_i)$ in $(\P(\bigoplus_{i=0}^N V_i), d)$ generated by $\bigoplus_{i=1}^N V_i$. Let $\sigma \in I(\bigoplus_{i=1}^N V_i)$ be a cocycle.  A \emph{defining system} for $\sigma$ is a $\P$-algebra map
    $$
    g_N: (\P(\bigoplus_{i=0}^N V_i), d) \to A
    $$
    such that $H^*(g_N)|_{V_0} = H^*(f)|_{V_0}.$ The \emph{$\sigma$-cotriple product set} is given by the collection of $H^*(g_N)(\sigma)$ where $g_N$ ranges across all choices of defining systems.
\end{definition}
\begin{remark}
    Both definitions have advantages. The spectral sequence definition immediately shows that cotriple products are homotopy invariant. The Sullivan algebra definition shows that Massey products are examples of cotriple products.
\end{remark}
Rather like usual Massey products, one only needs to compute the $\sigma$-cotriple product set rather than the indeteminacy in the spectral sequence to use cotriple products.
\begin{proposition}
\label{productsetease}
	 Let $\P$ be an operad that reflects homotopy equivalences. A morphism  of $\mathcal P$-algebras $f:A\to B$ preserves cotriple product sets. 
	If furthermore $f$ is a quasi-isomorphism, 
	then $f_*$ induces a bijection between the corresponding cotriple product sets.
\end{proposition}
\begin{proof}
    The first statement is straightforward as one can verify that given a $\sigma$-defining system
    $$
    g_N: (\P(\bigoplus_{i=0}^N V_i), d) \to A,
    $$
    postcomposing by $f$ gives a defining system $h_N$ on $B$
    such that $f^\ast H^*(g_N)|_{V_0} = H^*(h_N)|_{V_0}.$ Therefore the $\sigma$-cotriple product set on $B$ is a subset of that on $A.$

    \medskip 

    To prove the second statement, first observe that if the quasi-isomorphism $f: A \to B$ is surjective, one may lift any defining system on $g_N: (\P(\bigoplus_{i=0}^N V_i), d) \to B$ via the algorithm of Theorem \ref{thm:llp}. The general case follows from a similar argument to Proposition \ref{prop:sullivan_algebras_respect_weak_eq}, as in the proof of that result, one may replace the map $f$ with a zig-zag of quasi-isomorphisms $A\xleftarrow{g}C \xrightarrow{h} B$ where $B$ is surjective. One can then lift the defining system to $C$ and then push it forward to $A.$
\end{proof}
\begin{remark}
    This definition allows us to extend the notion of cotriple products to algebras over other operads, like the commutative operad, where $\P$ is not a  functor that reflects homotopy equivalences. In this case, we cannot necessarily deduce the homotopy invariance of such products automatically (and indeed, frequently they will not be, see Example \ref{counterexampletohomotopytype}). Therefore, each time we define such a product, we shall need to manually check homotopy invariance.
\end{remark}
The next theorem essentially states that this formulation of cotriple products is equivalent to the spectral sequence one.
\begin{theorem}
\label{Thm : Differentials EMSS and cotriple products}
Let $\P$ be an operad that reflects homotopy equivalences. Let $A$ be a $\P$-algebra and fix a choice of $f: (\P (\bigoplus_{i=0}^N V_i), d) \xrightarrow{\sim} A$ a $N$-step Sullivan model for $A.$  Let $\sigma \in I(\bigoplus_{i=1}^N V_i)$ be a cocycle.
Then there exists an element 
\[
G(\sigma) \in \P^{\circ N}(H^*(A))
\]
which survives to the $E_N$-term of the $\P$-cotriple spectral sequence,
and 
$$
d_{N-1}\left([G(\sigma])\right) \in (-1)^{N-2} \left[\idd\otimes H^*(f(\sigma))\right].
$$
\end{theorem}
To prove this we shall make use of the Staircase Lemma \cite[Lemma 2.1]{kraines72}, 
which we briefly recall next.

\begin{lemma} 
\label{lemma : Staircase lemma1}
Let $A=\left(A_{*,*}, d', d''\right)$ be a bicomplex,
denote by $d$ the differential on its total complex, 
and fix $c_1,\dots, c_n$ homogeneous elements in $A$. 
Suppose that $d'c_s = d''c_{s+1}$ for $1 \leq s\leq n-1$,
and define $c := c_1-c_2+\cdots +(-1)^{n-1}c_n.$ 
Then, $dc = d'c+d''c = d''c_1+ (-1)^{n-1}d'c_n,$ and furthermore,
in the spectral sequence $\left\{\left(E^r,d^r\right)\right\}$ associated to the bicomplex, 
if $d''c_1=0$ then $c_1$ survives to $E^n$, and $d^n[c_1]= (-1)^{n-1}[d'c_n].$
\end{lemma}

Our approach to proving Theorem \ref{Thm : Differentials EMSS and cotriple products} is therefore to construct a sequence $c_1,\dots c_{r-1}$ 
satisfying the conditions of the Staircase Lemma.

\begin{proof}[Proof of Theorem \ref{Thm : Differentials EMSS and cotriple products}]
Our first step is to recursively define a sequence $x_0, \dots, x_N$ where $$x_i\in \P^{\circ i}\left(\bigoplus_{j=0}^{r(i)} V_j\right)$$
for some $r(i) \leq k.$ Firstly, let $x_0 = \sigma  \in \P(\bigoplus_{i=0}^k V_i) = \P^{\circ 1}(\bigoplus_{i=0}^k V_i)$. Then one obtains an element of $x_1 \in \P^{\circ 2}(\bigoplus_{i=0}^{k-1} V_i)$ by replacing every occurrence of an element in $v_n\in V_k$ by the formula for $dv_n \in \P(\bigoplus_{i=0}^{k-1} V_i)$ and identifying $\P\circ \P(\bigoplus_{i=0}^{k-1} V_i) = \P^{\circ 2}(\bigoplus_{i=0}^{k-1} V_i).$  Continuing this procedure, we obtain $x_i$ for $i = 0,\cdots k.$ There are maps
$$
g_k: \P^{\circ l}(\bigoplus_{i=0}^{k-l} V_i) \to \P^{\circ l}(A)
$$
defined on generators $v\in V_i$ by
$$
v \mapsto f(v).
$$
We define $c_i = g_{n-i-1}(x_{n-i-1})$. The element $G(\sigma)$ in the proof statement is $d(c_0).$
    
To finish, we must verify that the conditions of the Staircase Lemma \ref{lemma : Staircase lemma1} are met.
Denote by $\partial$  the external differential on $\mathcal P(A)$, and by $d^*$ its internal differential.
Then, since $dv = 0$ for each $v\in V_0$, it follows that $d^* c_1 = 0$.
A routine calculation  shows that $d^* c_{s+1} = \partial c_s$ for each $s$.
It follows from the Staircase Lemma that 
\[
d_{n-1}[c_1] = (-1)^{n}[\partial c_{n-1}] = (-1)^{n}[f(\sigma)].
\]
This concludes the proof.
\end{proof}
\section{Cotriple products for strictly commutative dg-algebras}
In this section, we shall apply the theory of cotriple products developed in the last section to the case of strictly commutative algebras.
\subsection{Secondary cotriple products}
\label{froben123}
\label{compute}
Fundamentally, the main difference between the rational and $p$-adic commutative dg-algebras is that the $\Sym$ functor does not behave well homotopically in positive characteristic. 
\begin{example}
\label{ex:symbadlybehaved}
    The functor $\Sym: \mathsf{dg-}R\mathsf{-mod} \to \mathsf{dg-}R\mathsf{-Com-alg}$ does not reflect homotopy equivalence. For example, when $R= \mathbb F_p,$ the dg-modules $V=0$ and $W=  [\mathbb F_p x\to \mathbb F_p dx]$ are homotopy equivalent. However $H^*(\Sym(V)) = 0$, while $H^*(W)$ is non-zero as $x^p$ represents a nontrivial cohomology class.
\end{example}
The reader should be warned that the theory of commutative dg-algebras in positive characteristic has more subtleties than the rational case. For example,  commutative dg-algebras over $\mathbb F_p$ possess \emph{higher commutative  operations}, that is cotriple products that do not arise in the same way as classical higher Massey products of \cite{massey58} and \cite{flynn23}. We shall use these to produce some examples; notably examples of dg-algebras over $\mathbb Z$ with torsion-free cohomology that are formal over $\mathbb Q$ but not over $\mathbb F_p$. The first examples of these is the following.
\begin{definition}
\label{def:higher_commutative}
    Let $A$ be a commutative dg-algebra over $\mathbb F_p$. Let $x,y\in H^*(A)$ be homogeneous elements such that $xy = 0$. Choose cocycles $a,b\in A$ representing $x,y$ respectively. Then there exists $c\in A$ such that $dc = a b .$ Then $c^p$ is a cocycle which we call the \emph{type 1 secondary Frobenius product} of $x$ and $y$. Like classical Massey products, this operation has indeterminancy. Indeed, if we choose another lift $dc' = a b$ it must differ from the element $c$ by a cocycle $\sigma \in Z^{|x|+|y|-1}$. So $c'^p = c^p + \sigma^p.$ So, for $p=2,$ $c^p$ represents a well defined element of 
$$
\frac{H^{p(|x|+|y|-1)}(A)}{H^{(|x|+|y|-1)}(A)^p+x^pH^{p(|y|-1)}(A)+y^pH^{p(|x|-1)}(A)}
$$
where the term $x^pH^{p(|y|-1)}(A)+y^pH^{p(|x|-1)}(A)$ in the denominator accounts for the choice of representatives $x$ and $y.$  If the prime $p$ odd, assume without loss of generality, that $|x|$ is even and $|y|$ is odd. Then $c^p$ is a well-defined class of
$$
\frac{H^{p(|x|+|y|-1)}(A)}{H^{(|x|+|y|-1)}(A)^p+y^pH^{p(|x|-1)}(A)}.
$$
\end{definition}
\begin{remark}
    The striking difference between Type 1 Frobenius operations and ordinary Massey products in characteristic 0 is the dependence of the indeterminacy on the initial choice of cocycles. For a concrete example of this in practice, see Example \ref{counterexampletohomotopytype}. This is an added  complication with the development of cotriple products in positive characteristic. The underlying reason for this problem is very simple, as an $E_\infty$-algebra over $\mathbb F_2$, $c\otimes c$ is not always a cocycle, but if we work with the cup-1 algebras defined in Subsection \ref{camposexample}
    $$
    c\otimes c + c\cup_1 (a\otimes b) + a^{\otimes 2}\otimes K + L \otimes b^{\otimes 2}
    $$
    is, where $a,b,c$ are as above, $dL = a\cup_1 a$ and $dK = b \cup_1 b.$ The extra term comes from the fact that, in a cup-1-algebra, one can add cocycles to $K$ and $L.$
\end{remark}
For odd primes, there is a second type of secondary cohomology operation on commutative dg-algebras.
\begin{definition}
Let $p$ be an odd prime and $A$ be a commutative dg-algebra over $\mathbb F_p$. Then there is a \emph{type 2 secondary Frobenius product} defined for homogeneous elements $x,y\in H^*(A)$ such that $xy = 0$. Choose cocycles $a,b\in A$ representing $x,y$ respectively. Then there exists $c\in A$ such that $dc = x y.$ Then, it follows from the graded commutativity of multiplication that $c^{p-1}ab$ is a cocycle which we call the \emph{type 2 secondary Frobenius product} of $x$ and $y$. In this case, the operation represents a well-defined element of 
$$
\frac{H^{(p-1)(|x|+|y|-1)+|x|+|y|}(A)}{H^{(|x|+|y|-1)}(A)^{p-1}xy}
$$
\end{definition}
Such operations do not exist over $\mathbb F_2$ as there is no reason for $c^{p-1}ab$ to be a cocycle. This is because the relation $x^2 = 0$ for $|x|$ odd does not hold for commutative dg-algebras in characteristic 2. 
\begin{remark}
\label{obstructiontoDP}
    Observe that $d(\frac{1}{p}c^p) = c^{p-1}ab.$ Therefore type 2 secondary Frobenius products vanish on divided power algebras. Therefore this kind of operation provides an obstruction for a commutative dg-algebra $A$ to be weakly equivalent to a divided power algebra. It is obvious that non-zero type 1 Frobenius operations are also obstructions as $(-)^p$ vanishes on divided power algebras.
\end{remark}
The next lemma is a short verification that our operation is well defined.
\begin{lemma}
\label{lem:liftindependent}
    Up to indeterminacy, the secondary Type 1 and Type 2 Frobenius products of $x, y \in H^*(A)$ do not depend on the choice of cocycles representing $x$ or $y$. 
\end{lemma}
\begin{proof}
  Let $a'$ and $b'$ respectively be an alternative choice of cocycles representing $x, y$. Then $a' - a =dK$ and $b' - b = dL$ are coboundaries.
    Let 
$
c' = c + aL + b'K 
$.
    Then we have $dc' = a'b'$. Moreover, we have $(c')^p = c^p + a^p L^p + (b')^p K^p.$ Observe  that $K^p$ and $L^p$ are cocycles and therefore represent elements of $H^{p(|x|-1)}(A)$ and $H^{p(|y|-1)}(A)$ respectively. We therefore have that $c^p$ and $(c')^p$ represent the same element of 
$$
\frac{H^{p(|x|+|y|-1)}(A)}{H^{(|x|+|y|-1)}(A)^p+x^pH^{p(|y|-1)}(A)+y^pH^{p(|x|-1)}(A)}.
$$
If $p$ is odd, it follows that $|K|$ is odd, and we have that $c^p$ and $(c')^p$ represent the same element of 
$$
\frac{H^{p(|x|+|y|-1)}(A)}{H^{(|x|+|y|-1)}(A)^p+y^pH^{p(|x|-1)}(A)}.
$$
   Now we consider the case of Type 2 operations.  Let $a',b', c', L, K$ all be defined as before. Then we have
   $$
   (c')^{p-1}a'b' = (c +aM+ MdL+bL)^{p-1}(dL+a)(dM+b)
   $$
   This can be written as 
   $$
   (c)^{p-1}ab + d\left(\sum_{\substack{i+j = p \\i,j\neq 0}} \frac{1}{i} \binom{p-1}{i} c^i\left(aM+ MdL+bL\right)^j \right)
   $$
   We have therefore that $(c')^{p-1}a'b'$ and $(c)^{p-1}ab$ represent the same element of
   $$
\frac{H^{p^n(|x|+|y|-1)+|x|+|y|}(A)}{H^{(|x|+|y|-1)}(A)^{p^n-1}}.
$$
This proves the lemma.
\end{proof}
The following lemma is straightforward consequence of the previous one.
\begin{lemma}
\label{lem:homotopyinvariance}
   Secondary Frobenius products are homotopy invariant. That is, 
	if $x, y \in H^*(A_1)$ are homogeneous elements such that
 their secondary Frobenius product $z$ is defined then for any zig-zag of quasi-isomorphisms
 $$
 \begin{tikzcd}
     A_1 \arrow[r, "f_1"] & A_2 &\arrow[l, "f_2", swap]\cdots \arrow[r, "f_{n-1}"] & A_n
 \end{tikzcd}
 $$
 the Massey product of $f_1^*(f_2^*)^{-1}\cdots f_{n-1}^*(x)$ and $f_1^*(f_2^*)^{-1}\cdots f_{n-1}^*(y)$ is, up to indeterminacy, equal to $f_1^*(f_2^*)^{-1}\cdots f_{n-1}^*(z)$ in $H^*(A_n)$.
\end{lemma}
The most immediate application of these cohomology operations is that they provide extra obstructions when comparing commutative dg-algebras. Our first example is an algebra with no torsion in its cohomology that is formal over $\mathbb Q$ but not $\mathbb F_2.$
\begin{example}
\label{ex:algebrasweaklyhomotopicoverQbutnotF2}
    Consider the case $p=2$ and consider the following dg algebras over $\mathbb Z$. $$A= \operatorname{Sym}(x, y, z) / (xy, xz, yz) \qquad B =  \operatorname{Sym}(x, y, z, t) / (t^p - z, tz, xz, yz, t^{p+1}, t^{p-1}x, t^{p-1}y)$$
where $x,y,z$ are cocycles, we have $dt = xy$ and $| x| = |t| = 2,$ $| y| = 1$ and $|z| = 2p.$ 

A set of generators for $A$ as a free $\mathbb Z$-module is given by $\{x, x^2, x^3, \cdots, y, z, z^2, z^3, \cdots\}$ and its cohomology ring is $A$ itself. For $B$, one has a set of generators given by
$$\{ z^i \mbox{for } i \geq 1\}\cup\{ t^ix^jy^k, \mbox{for }  p-2 \geq i\geq 1, j  \geq 0 \mbox{ and } k \in \{0,1\}\}
$$
Therefore, by direct computation, one has that the cohomology of $B$ is equal to $A$. However, by Lemma \ref{lem:homotopyinvariance}, they are not quasi-isomorphic as commutative dg-algebras as all the secondary Frobenius products in $A$ vanish, while in $B$ the secondary Frobenius product of $x$ and $y$ is $\{z\}.$ However, these algebras are quasi-isomorphic over $\mathbb Q$ via the zig-zag
$$
A \xleftarrow{f} \operatorname{Sym}(x, y, z, t) / (xz, yz) \xrightarrow{g} B
$$
where $dt = xy$ and $f,g$ are the obvious projection maps. So $B$ is formal over $\mathbb Q$ but not $\mathbb F_2$.
\end{example}
\begin{remark}
    Formal algebras have vanishing Massey products. The previous example therefore demonstrates that Frobenius products are a different set of invariants to classical Massey products.
\end{remark}
  The problem with extending Example \ref{ex:algebrasweaklyhomotopicoverQbutnotF2} to $\mathbb F_p$ directly is that the element $ty\in B$ becomes the Massey product $<x,y,y>$ as $y^2 = 0$. We therefore alter it slightly to produce an example of two algebras with the same cohomology that are quasi-isomorphic over $\mathbb Q$ but not $\mathbb F_p$

 \begin{example}
\label{ex:algebrasweaklyhomotopicoverQbutnotFp}
    Consider the following dg algebra over $\mathbb Z$. \begin{align*}
    A &= \operatorname{Sym}(x, y, z, t) 
    \Big/ \big( x, y, z^2, xz, yz, tz, t^p, t^{p-1}x, t^{p-1}y \big).
\end{align*}
    where $dt = xy$,  $| x| = |t| = 2,$ $| y| = 1$ and $|z| = 2p.$. A  basis for $A$ is given by 
    \begin{align*}
    S &= \{x, x^2, x^3, \dots, y, z\} 
    \cup \{ t^i x^j y^k \mid p > i \geq 0, j > 0, k \in \{0,1\} \}.
\end{align*}
    With coefficients in either $\mathbb Q$ or $\mathbb F_p$, the cohomology of $A$ is therefore
\begin{align*}
    \operatorname{Sym}(x,y,z) \Big/ \big( xy, yz, xz \big) 
    \cup \big\{ s_1, s_2, \dots, s_{p-1} \big\}.
\end{align*}
    where the added elements are Massey products $s_i = t^iy = <x, y, s_{i-1}>$ in $A$. 

    Our second algebra is
\begin{align*}
    B &= \operatorname{Sym}(x, y, z, t) 
    \Big/ \big( t^p - z, tz, xz, yz, t^{p+1}, t^{p-1}x, t^{p-1}y \big).
\end{align*}
where $x,y,z$ are cocycles, we have $dt = xy$ and $ds = ty$ and $|s| = | x| = |t| = 2,$ $| y| = 1$ and $|z| = 2p.$

For $B$, one has a basis given by
\begin{align*}
    S &= \{ z^i \mid i \geq 1 \} 
    \cup \{ t^i x^j y^k \mid p-2 \geq i \geq 1, \, j, k \geq 0 \}.
\end{align*}
Therefore, by direct computation, one has that the cohomology of $B$ is equal to $A$. However, by Lemma \ref{lem:homotopyinvariance}, they are not quasi-isomorphic as commutative dg-algebras as all the secondary Frobenius products in $A$ vanish, while in $B$ the secondary Frobenius product of $x$ and $y$ is $\{z\}.$ However, these algebras are quasi-isomorphic over $\mathbb Q$ via the zig-zag
\begin{align*}
    A \xleftarrow{f} \operatorname{Sym}(x, y, z, t) \Big/ \big( xz, yz \big) 
    \xrightarrow{g} B.
\end{align*}
where $dt = xy$ and $f,g$ are the obvious projection maps.
\end{example}

Lastly, we give a counterexample constructed using a Type 2 Frobenius product.

\begin{example}
\label{ex:algebranoteqtodividedpower}
    Here we give an example of an algebra that has a divided power structure on its cohomology is nonetheless not quasi-isomorphic to a divided power algebra. For this we use type 2 Frobenius products.  Consider the following dg-algebras
    \begin{align*}
    A &= \operatorname{Sym}(x, y, z, t) \Big/ \big( xz, yz, zs_i, s_{p-1}x, s_{p-1}y, s_i s_j - x s_{i+j}, s_p x, s_p y \big) \\
    B &= \operatorname{Sym} \big( \mathbb{F}_p \langle x, y, z  \rangle, t \big) 
    \Big/ \big( t^p, t^{p-1} y - z, t x^2 \big).
\end{align*}
    where $dt=xy$ and the degrees $|x|,|t|$ are even and $|y|, |z|$ are odd. The cohomology of this is given by $(\mathbb F_p\langle x, y, z, c_1,c_2, \dots c_{p-2} \rangle/(xy, c_ic_j, xc_i,yc_i, zc_i))$, where the $c_i = t^iy$, which is a divided powers algebra.  Nonetheless, the  type 2 Frobenius product  of $x,y$ is $xz$ so by Remark \ref{obstructiontoDP} it cannot be quasi-isomorphic to a commutative algebra that can be equipped with the structure of a divided powers algebra.
 \end{example}

We conclude this section with a brief completeness result.
\begin{definition}
    We call a cotriple product \emph{primitive} if it arises from monomial relations in cohomology.
\end{definition}
\begin{proposition}
    All secondary primitive cotriple products on a commutative dg-algebra $A$ over $\mathbb F_p$ are linear combinations of
    \begin{itemize}
        \item classical Massey products.
        \item Type 1 secondary Frobenius operations
        \item Type 2 secondary Frobenius operations.
        
    \end{itemize}
\end{proposition}
\begin{proof}
    Let $\Sym(V_0\oplus V_1)\to A$ be a step 2 Sullivan resolution. We recall that secondary cohomology operations are precisely given by terms representing cohomology in $I(\Sym(V_1)).$ These can be directly verified to be linear combinations of elements of the following form
    $$
    x^{p^n}, \qquad x^{p^n-1}dx, \qquad uc-av 
    $$
    where $x, u,v \in V_1$ and $du = ab$ and $dv = bc$ for $a,b,c \in \Sym(V_0)$. These corresponds to type 1 and type 2 Frobenius products and classical Massey products respectively. 
\end{proof}
\subsection{The relationship between associative and commutative algebras}
\label{camposexample}
    In \cite{campos23}, the authors raise the question of whether two commutative dg-algebras are quasi-isomorphic as associative dg algebras if and only if they are also quasi-isomorphic as commutative dg algebras. They prove the following theorem.
    \begin{theorem}\cite[Theorem A]{campos23}
        Let $A$ and $B$ be two commutative dg algebras over a field of characteristic zero. Then, $A$ and $B$ are quasi-isomorphic as associative dg algebras if and only if they are also quasi-isomorphic as commutative dg algebras. 
    \end{theorem} 
        We give a counterexample this in positive characteristic using Frobenius product obstruction theory. Similar examples should hold at all primes.  More precisely, our statement is the following.
        \begin{theorem}
        \label{mycounterexample}
            There exists two commutative dg algebras $A$ and $B$ over a field of characteristic two which may be distingushed via their type 1 Frobenius operation. Nonetheless, there exists an associative algebra $C$ such that there is a zig-zag of associative weak equivalences
            $$
            A \xleftarrow{\sim} C \xrightarrow{\sim} B
            $$
        \end{theorem}
        Since it must have commutative cohomology, such an algebra $C$ must be commutative up to homotopy (in the most naive sense possible). We shall capture this idea using the notion of a lax cup-1-algebra which we introduce in subsubsection \ref{cup-1_algebra}. The key idea is that type 1 Frobenius products may be defined on such algebras but they have a different indeterminacy there than on commutative algebras; we study this phenomenon in Proposition \ref{indeterminacy}. In subsubsection \ref{thecomalgebras}, we define the commutative algebras $A$ and $B.$ In subsection $C,$ we define the lax cup-1-algebra $C$ and compute its cohomology. Finally in subsubsection \ref{thezigzag}, we give the maps appearing in the zig-zag in Theorem \ref{mycounterexample}.
    \subsubsection{Cup-1-algebras}
    \label{cup-1_algebra}
    We recall the following flavour of algebra. They are essentially commutative algebras up to homotopy (but not coherently) and are similar to those appearing in \cite{porter}.
\begin{definition}
    A \emph{lax cup-1-algebra} is a chain complex $A$ equipped with two binary operations $\cup$ and $\cup_1$. The operation $-\cup-$ is degree 0 and associative. The second $-\cup_1 - $ is degree $-1$ and associative. These are intertwinned by two identities. First we have the Hirsch identity, namely that
     \begin{equation}
     (u\cup v)\cup_1 w = u \cup(v\cup_1 w) + (u\cup_1 w)\cup v
     \end{equation}
     Secondly, we have the \emph{Steenrod relation}
     \begin{equation}
     d(u\cup_1 v) = (du\cup_1 v) +(u\cup_1 dv) +(u\cup v)+ (v\cup u)
     \end{equation}
     A cup-1-algebra is described as \emph{strict} if $\cup_1$ is graded commutative. A morphism of cup-1-algebras is a morphisms that preserves both operations. 
    \end{definition}
    \begin{lemma}
    A strict cup-1-algebra structure on $A$ extends to an $\E$-algebra structure.
\end{lemma}
\begin{proof}
    In the surjection operad, which is a quotient of $\E$ defined in \cite{berger04}, $\cup$ corresponds to the operation $(1,2) \in \mathcal{X}(2)_0$ and $\cup_1$ corresponds to  $(1,2,1) \in \mathcal{X}(2)_1$. Consider the quotient operad of $\mathcal X$ given by quotienting by all operations not generated by these operations. Since $ (1,2,1, 2) \in \mathcal{X}(2)_2$ vanishes, we obtain the commutativity of $\cup_1$. The Steenrod and Hirsch relations can now be obtained by routine computations.
\end{proof}
\begin{remark}
 Strictly commutative algebras are examples of cup-1-algebras such that the $\cup_1$ operation is identically 0. The Steenrod relation then ensures strict commutativity.
\end{remark}
The following definition will be useful for our later computations.
\begin{definition}
    Let $U = \operatorname{Cup}(X)/(R)$ be a (lax) cup-1-algebra presented in terms of generators and relations. Let $m$ be a monomial in $A$, constructed from the generators using both $\cup_1$ and $\cup$. Then $m$ is \emph{reduced} if it is written as 
    $$
    m = m_1 \cup m_2\cup \cdots \cup m_n
    $$
    where each $m_i$ is a monomial constructed only using the $\cup_1$ operation. 
\end{definition}
Clearly, the free cup-1-algebra has a basis consisting of reduced monomials.
\subsubsection{Frobenius products for cup-1-algebras}
\label{sec:frobcup1}
Next, we describe the Frobenius products that exist on cup-1-algebras.
\begin{proposition}
\label{indeterminacy}
    Let $A$ be a cup-1-algebra that is quasi-isomorphic to a strictly commutative dg-algebra. Then the following is a cocycle
    $$
    c\cup c + c\cup_1(a\cup b) + K
    $$
    where $dc = a\cup b$, $dK = (a\cup b)\cup_1 (a\cup b)$. If, furthermore, $A$
 is a strict cup-1-algebra, then this operation is equivalent to 
 $$
    c\cup c + c\cup_1(a\cup b) +a^2 \cup K' + L' \cup b^2.
    $$
    where $dK' = b\cup_1 b$ and $dL' = a\cup_1 a$. 
 \end{proposition}
\begin{remark}
    The equivalence of $A$ to a strictly commutative dg-algebra guarantees the existence of $K$ and $L$ as $b\cup_1 b$ and $a\cup_1 a$ represent Steenrod operations and these must vanish for a strictly commutative algebra.
\end{remark}
\begin{proof}
    The proof is a completely straightforrward computation which is nonetheless pedagogic.
    $$
    d c\cup c = (a\cup b) \cup c + c\cup (a\cup b) 
    $$
    This is nonzero as $\cup$ is not commutative. Next, one has, by the Steenrod identity
    $$
    d(c\cup_1(a\cup b)) = (a\cup b) \cup c + c\cup (a\cup b) + (a\cup b) \cup_1(a\cup b)
    $$
    Finally one has
    $$
    d\left(a^2 \cup K + L \cup b^2 \right) = a \cup (b\cup_1 (a\cup b)) + (a\cup_1 (a\cup b)) \cup b = (a\cup b) \cup_1(a\cup b)
    $$
    or, if $A$ is strict, this can be written as 
    $$
    d\left(a^2 \cup K' + L' \cup b^2 \right) = a^2 \cup (b\cup_1 b) + (a\cup_1 a) \cup b^2 = (a\cup b) \cup_1(a\cup b)
    $$
    where the last equality follows from the Hirsch identity applied twice. Summing all three expressions one gets zero; which proves that the expression is a cocycle.
\end{proof}
The expression above allows one to compute the indeterminacy of this higher cup-1-product operation which can be computed on strict algebras to be
$$
\frac{H^{2(|x|+|y|-1)}(A)}{H^{(|x|+|y|-1)}(A)^2+x^2H^{2(|y|-1)}(A)+y^2H^{2(|x|-1)}(A)}.
$$
In other words, it is the same as in the commutative case. The main difference here is that, in this case, they fill out the complete indeterminacy.

\medskip

The point of using lax cup-1-algebras however, is that $K$ does not split apart in the same way as in the strict case. This means that the operation has larger indeterminacy.
\subsubsection{The commutative algebras}
\label{thecomalgebras}
     In this subsubsection, we shall construct the commutative algebras $A$ and $B$ from Theorem \ref{mycounterexample}. Consider the following strictly commutative dg-algebras over $\mathbb F_2$,
     $$
         A = \Sym(x,y, z)/(x^3, y^3, xy, z^2, x^2y, yx^2, x^2y, yx^2,yz, z^2 )
    $$
     where $|x|= |y|= 2$, $|z| = 4$. The algebra $A$ has the following linear basis $\{x,x^2, y, y^2, z, xz, a\}$ and coincides with its cohomology.
     Then we have
     $$B = \Sym(x,y, z, t)/(x^3, y^3, az, tx^iy^j, t^3, t^2-xy  z^2, x^2y, yx^2, x^2y, yx^2, az,yz, z^2, ax,ay, at a^2 )$$
     such that $|t| = 3$ and $dt = xy$ and where $i,j$ range over the positive integers such that $i+j = 2$. It is easy to explicitly write down a basis for $B$ as $$\{x,x^2, y, y^2, z, xz, xy, x^2y, xy^2, x^2y^2, t, tx, ty, txy\}$$ and one then easily verifies that its cohomology is equal to $A$. Finally we have
     $$B = \Sym(x,y, z, t)/(x^3, y^3, az, tx^iy^j, t^3, t^2+xy  z^2, x^2y, yx^2, x^2y, yx^2, az,yz, z^2, ax,ay, at a^2 )$$

     \medskip
     
     Next one computes the Frobenius operation of $x$ and $x^2$. These operation is clearly strictly defined. Moreover, for $A$ it is \{0\}, but for $B$ it is $\{xy\}.$ These are both in different indeterminacy classes, so it follows that the algebras cannot be quasi-isomorphic either as commutative algebras. 
     
     \subsubsection{The cup-1-algebra $C$}
     \label{petersen}
   In this subsubsection, we shall construct the associative algebra $C$ from Theorem \ref{mycounterexample}. Consider the lax cup-1-algebra $C$ generated by the elements $x, y, t$ subject to the following
        \begin{itemize}
            \item $x, y, z $ are cocycles, $dt = x\cup y$
            \item 
            One has $|x|= |y|= 2$, $|z| = 4$, $|t| = 3$. 
            \item 
            We set  $S = x\cup_1 t + t\cup_1 x$ and  $T = y\cup_1 t + t\cup_1 y$. Clearly  $dT = (y\cup_1 xy), dS_2 = (x\cup_1 xy)$. This seem unnecessarily confusing now, but will greatly simplify the notation for describing the maps.
            \item
            We quotient out by every monomial containing $z$, except $z, x\cup z$ and $z\cup x$; and we also quotient by $x\cup z + y\cup z$.
            \item 
            $x$ and $y$ commute and we quotient out by $x\cup_1 x, x\cup_1 y, y\cup_1 x, y\cup_1 y$.
            \item Ignoring $z$, we introduce a new degree called \emph{$x$-word length}, denoted $|-|_x$ where $|y|_x= 0,$ $|x|_x = |T|_x = 1, |t|_x = 1, |S| = 2$. Similarly, we have \emph{$y$-word length}, denoted $|-|_y$ where $|x|_y  =0,$ $|y|_y = |S|_y =  1, |t|_y = 1$ and $|T|_y = 2. $ We consider word-length to be additive under both $\cup$ and $\cup_1$. The \emph{total word length} is the sum of the word lengths. The differential can easily be checked to preserve word length. We kill all monomials of $x$- or $y$-word length 3 or greater.
            \item 
            Finally, we impose the relation that $t\cup t + t\cup xy + x\cup T + y\cup S = 0$
\end{itemize}
In the proof of the following proposition, we also present a linear basis for this algebra with 32 elements. This is probably significantly easier to parse.
\begin{proposition}
    The cohomology of $C$ is equal to
         $$
         \Sym(x,y, z, a)/(x^3, y^3, xy, z^2, x^2y, yx^2, x^2y, yx^2, az,yz, z^2, ax,ay, a^2 )
    $$
     where $|x|= |y|= 2$, $|z| = 4$ and $|a| = 6$ and, in our previous notation, where $a = t\cup t + t\cup x^2 + L_1 +L_2.$
         \end{proposition}
         \begin{proof}
         Observe that the differential preserves $x$- or $y$-word-length and that our relations are homogeneous in word length since they are monomial. It follows that every cocycles can be written as the sum of cocycles that are homogeneous in word length. Therefore, if $C = \bigoplus_{i =1}^6C_i$ then $H^\ast(C) = \bigoplus_{i =1}^6 H^\ast(C_i)$. We proceed by computing a basis of reduced monomials.

\medskip
             
             Before proceeding further in the calculation, we make the following observation: one always has a linear generating set consisting of reduced monomials. We can also ignore the monomials containing $z$ as there is only three of them.  Therefore we can do a direct computation.

             \medskip

             First of all, one has that $C_1 = \mathbb F_2 x.$  It is completely straightforward to directly check the following by hand on reduced monomials.

             \medskip
\noindent\textbf{Total word length 1}
\begin{center}
\begin{tabular}{ |c|c| } 
 \hline
 Degree & Basis\\
 \hline
 2 & $x, y$ \\
 \hline
\end{tabular}
\end{center}
             
             \medskip

\noindent\textbf{Total word length 2}
\begin{center}
\begin{tabular}{ |c|c| } 
 \hline
 Degree & Basis\\
 \hline
 3 & $t$  \\ 
 4 & $x^2, xy, y^2$  \\
 \hline
\end{tabular}
\end{center}

\medskip

             \medskip

\noindent\textbf{Total word length 3}
\begin{center}
\begin{tabular}{ |c|c| } 
 \hline
 Degree & Basis\\
 \hline
 4 & $x\cup_1 t, y\cup_1 t, T, S $  \\ 
 5 & $x\cup t, t\cup x, y\cup t, t\cup y, x\cup_1 (xy), y\cup_1 (xy)$  \\ 
 6 & $ x^2y, xy^2,$  \\
 \hline
\end{tabular}
\end{center}
\noindent\textbf{Total word length 4}
\begin{center}
\begin{tabular}{ |c|c| } 
 \hline
 Degree & Basis\\
 \hline
 6 & $x\cup T, T\cup x, S\cup y, (x\cup_1 t)\cup y, (y\cup_1 t)\cup x, x\cup (y\cup_1 t), t\cup_1t$  \\
 7 & $ xy \cup t, t\cup xy, x\cup t \cup y, y\cup t\cup x,  x\cup (y\cup_1 xy),  (y\cup_1 xy)\cup x, (x\cup_1 xy)\cup y, t\cup_1 xy$  \\ 
 8 & $x^2y^2$  \\
 \hline
\end{tabular}
\end{center}
\medskip 
In the last table,  $t\cup t = t\cup xy + x\cup S + T\cup y$. We have also used that
$$
x\cup T+ y\cup S+ T\cup x = y \cup S 
$$
$$
(x\cup_1 t)\cup y+ (y\cup_1 t)\cup x+  x\cup (y\cup_1 t)  = y\cup( x\cup_1 t) 
$$
which come from the Hirsch identities.
\end{proof}
\subsubsection{The zig-zag}
\label{thezigzag}
     In this subsubsection, we shall construct the weak equivalences in the zig-zag from the statement of Theorem \ref{mycounterexample}. Clearly, the algebra $C$ is associative with respect to cup product $\cup$ by the definition of lax cup-1 algebras. The first map 
     $$
     f: C\to A
     $$
     is given by sending $x, y, z$ to themselves, and $t,S,T$ to $0$. We send all elements in the basis we computed that contain a $\cup_1$ to 0. This map is a quasi-isomorphism as it sends the generators of the cohomology to themselves.
    The other map is
     $$
     g: C\to B
     $$
     which is given by sending $x, y, z$ to themselves, $T \mapsto y$ and $t,S\mapsto 0$. We send all elements in the basis we computed that contain a $\cup_1$ to 0.  This map is also a quasi-isomorphism. The reader should note that these are maps of associative algebras, they are not maps of lax cup-1 algebras as $T$ and $S$ are not independent of $x$ and $y$ in the cup-1-algebras.

     \medskip 

     Since these maps are quasi-isomorphisms of associative algebras, we conclude that $A$ and $B$ are weakly equivalent as associative algebras but not as commutative algebras.

\subsection{Higher order cotriple products}
In this section, we shall define and study some families of primitive higher order cotriple products. As we saw in Subsection \ref{compute}, secondary cotriple products generally behave quite well for commutative dg-algebras. The only property that they lack is filling out the whole indeterminacy. Unfortunately, this failure then directly implies that  the higher order operations will not be definable on every choice of model. This then breaks the homotopy invariance of such operations.

\begin{example}
There are cotriple products of all orders. One way to see this is to choose a  vanishing $n^{th}$ order Massey product $m_n(x_1, x_2, \cdots, x_n) = dc$. Then $c^p$ is a Type 1 higher operation of order $n+1.$ 
\end{example}
The next example gives an explicit family arising directly from Type 1 Frobenius operations. The reader should be warned that we shall see shortly that this family is not always homotopy invariant.
\begin{definition}
\label{def:higher_commutative2}
    Let $A$ be a commutative dg-algebra over $\mathbb F_p$. Let $x,y\in H^*(A)$ be homogeneous elements such that $xy = 0$. A defining system for a \emph{$n^{th}$ order type 1 Frobenius product} is a collection $ \{a, b, c_1, \dots c_{n-1}\} $ such that $a, b$ are choices of cocycle representatives for $x, y$, $dc_1= ab$ and $c_i^p = dc_{i+1}$. The \emph{$n^{th}$ order type 1 Frobenius product} is then $c_n^p.$ In particular, second order type 1 Frobenius products coincide with those defined in Definition \ref{def:higher_commutative}.
\end{definition}
In order to be a useful class of operations it is important to compute the indeterminacy of the class. However, this is more complicated than it appears, because, in general, secondary cotriple products for the commutative operad do not completely fill out their indeterminacy.
\begin{example}
 For example, if it exists, the third order type 1 Frobenius product on a given algebra $A$ will be a well-defined element of 
$$
\frac{H^{4(|x|+|y|) - 6}(A)}{H^{2(|x|+|y|) - 3}(A)^2+x^4H^{4|y|-6}(A)+y^4H^{4|x|-6}(A)}.
$$
if there does not exist $u, v \in H^*(A)$, both nonzero, such that $x^2u + y^2v = 0.$ Otherwise each relation $x^2u + y^2v = 0$ in cohomology will give rise to extra secondary non-primary Frobenius operations in the obvious way. These will have some indeterminacy $X$ and the denominator of the above quotient will be $H^{2(|x|+|y|) - 3}(A)^2+x^4H^{4|y|-6}(A)+y^4H^{4|x|-6}(A) + X$.

This can be seen as follows: Firstly, we can add any choice of cocycle to $c_2$. This accounts for the $H^{2(|x|+|y|) - 3}(A)^2$ term.
Then let $a'$ and $b'$ respectively be an alternative choice of cocycles representing $x, y$. Then $a' - a =dK$ and $b' - b = dL$ are coboundaries.
    Let 
$
c_1' = c_1 + aL + b'K 
$.
    Then we have $dc_1' = a'b'$. Moreover, we have $(c_1')^2 = c^2 + a^2 L^2 + (b')^2 K^2.$ Suppose now that 
$$
dc_2 = c_1 
$$
then there exists $dc_2' =c_1'$ if and only if there exists an $R$ such that
$$
dR = a^2 L^2 + (b')^2 K^2.
$$
If there does not exist $u, v \in H^*(A)$, both nonzero, such that $x^2u + y^2v = 0$, we must have that the cocycles $L^2$ and $K^2$ are both zero in cohomology and hence, there exists $S, T$ such that
$$
dS = L^2 \qquad dT = K^2
$$
and therefore 
$$
R = a^2 S + (b')^2 T
$$
and therefore $(c_2+R)^2  = c_2^2 + a^4 S^2 + (b')^4 T^2$, from whence comes the $x^4H^{4|y|-6}(A)+y^4H^{4|x|-6}(A)$ term in the indeterminacy.
\end{example} 
However, a $n^{th}$ order Type 1 Frobenius product is not guaranteed to exist on commutative dg-algebra $A$ even if it does on other algebras weakly equivalent to $A$. This prevents Frobenius products from being homotopy invariant. This is a counterexample to \cite[Proposition 2.18]{flynn23} in positive characteristic.
\begin{example}
\label{counterexampletohomotopytype}
Consider the following dg-algebras over $\mathbb F_2$. $$A= \operatorname{Sym}(w, x, y, z) / (x^2w - z, xy, xz, yz)$$
$$ B =  \operatorname{Sym}(w, x, y, z, t) / (t^2 - z, xz, yz, t^{3}, y^2w - z,)$$
where $w, x, x', y,z$ are cocycles, we have $dt = xy$ and $ds =x -x'$ The cohomology ring of both algebras is $A.$ In this case, one can check that type 1 Frobenius product of $x$ and $y$ vanishes. Define
$$ C =  \operatorname{Sym}(w, x, x', y, z, s, t) / (t^2 - z, xz, yz, t^{3}, s^{3}, y^2w - z, s^2 - w).$$
There is a zig of quasi-isomorphisms
$$
\begin{tikzcd}
A \arrow[r, "\sim"] &C & \arrow[l, swap, "\sim"] B.
\end{tikzcd}
$$
The third order type 1 Frobenius operation associated to the relation $xy = 0$ in cohomology is defined in $A$ and $C$, as the second order product set is $\{ 0\}$ in $A$ and $\{0, z\}$ in $C$. However, it is not defined in $B$ as the second order type 1 Frobenius product set is $\{z\}$, which contains no coboundaries.
\end{example}
The examples above imply that we need an extra condition to ensure that higher order operations are homotopy invariant. One such condition is that the higher operation is  \emph{strictly defined}. This condition has the added benefit of allowing easy computation of the indeterminacy.
\begin{proposition}
\label{indeterminacyofhigheroperations}
    Let $A$ be a commutative dg-algebra and suppose $x,y \in H^*(A)$ are such that their type 1 $n^{th}$ Frobenius product is \emph{strictly defined}, that is, that $xy = 0$ and
    $$
    H^{p(|y|-1)}(A) = H^{p(|x|-1)}(A) = \{0\}
    $$
    $$
    \cdots
    $$
    $$
    H^{p^{n-1}(|y|) - \sum_{i=1}^{n-1} p^i}(A) = H^{p^{n-1}(|x|) - \sum_{i=1}^{n-1} p^i}(A) = \{0\}
    $$
    and the $(n-1)^{th}$ Frobenius product is equal to 0. Then $n^{th}$ order type 1 Frobenius product is defined and is a well-defined element of
    $$
\frac{H^{p^{n-1}(|x|+|y|) - \sum_{i=1}^{n-1}p^i}(A)}{H^{p^{n-1}(|x|+|y|) - 2^{n-1}+ \sum_{i=0}^{n-1} p^i}(A)^p+x^{p^{n-1}}H^{p^{n-1}(|y|) - \sum_{i=1}^{n-1} p^i}(A)+ y^{p^n}H^{p^{n-1}(|x|) - \sum_{i=1}^{n-1} p^i}(A)}
$$
and therefore is invariant under quasi-isomorphism.
\end{proposition}
\begin{proof}
 The proof is by induction on the order of the operation. First observe that if it is strictly defined, the secondary Frobenius operation represents a well-defined class of 
$$
\frac{H^{p(|x|+|y|-1)}(A)}{H^{(|x|+|y|-1)}(A)^p+x^pH^{p(|y|-1)}(A)+y^pH^{p(|x|-1)}(A)} = \frac{H^{p(|x|+|y|-1)}(A)}{H^{(|x|+|y|-1)}(A)^p}.
$$
Therefore it fills out its indeterminacy and the third order Frobenius operation is defined at every commutative algebra. Finally it has a well defined indeterminacy, which can be computed as follows . Recall from the proof of Lemma \ref{lem:liftindependent} that the Frobenius product set is given by  all elements of the form
$$
c_1^p + a^p L^p + (b')^p K^p
$$
(using the notation from the proof of that proof). By assumption $dc_2 = c_1^p.$ But, one can also choose $c_2' = c_2+\sigma$ such that
$$
d(c_2 + \sigma) = c_1^p + a^p L^p + (b')^p K^p
$$
This implies $d\sigma = a^p L^p + (b')^p K^p$. By the strictly defined hypothesis, all such terms are coboundaries. Moreover, both $a^p L^p$ and $(b')^p K^p$ are individually coboundaries. It follows that $\sigma$ can be factored as $\sigma' +\tau_1 +\tau_2 $, where $\sigma_1$ is a cocycle and 
$$d\tau_1 = a^p L^p \qquad d\tau_2 = (b')^p K^p.$$
\emph{A priori}, $\tau_1$ and $\tau_2$ are only defined up to cocycle, but any choice of cocycle can be added to $\sigma'$, so we may assume they are unique for any given choice of $a,a', b, b', L, K$. Then, since $H^{p(|y|-1)}(A) = H^{p(|x|-1)}(A) = \{0\},$ one has that $dR = L^p $ and $dS= K^p$. 
Therefore, we have that 
$$
c_2' = c_2 + \sigma' + a^p R  +b^p S.
$$
So 
$$
(c_2)^p =  c_2^p + (\sigma')^p + a^{p^2} R^p  +b^{p^2} S^p.
$$
As $R^p$ and $S^p$ are cocycles, one has the desired indeterminacy.  Therefore we have the desired invariance under quasi-isomorphism.

\medskip

Then, by induction, assume that  the order $k$ type 1 Frobenius product is defined and has the desired indeterminacy. Moreover assume that order $k$ type 1 Frobenius product set takes the form of a subset of:
\begin{multline*}
\big\{c_k^p +\sigma^p + a^{p^k}P^p + b^{p^k} Q^p: \sigma \in Z^{p^{k-1}(|x|+|y|) - \sum_{i=1}^{k-1} p^i}(A),  \\  P\in C^{p^{k-1}(|y|) - \sum_{i=1}^{k-1} p^i}(A), Q\in 
C^{p^{k-1}(|y|) - \sum_{i=1}^{k-1} p^i}(A) \big\}
\end{multline*}
for a fixed $c_k \in Z^{p^{k-1}(|y|) - \sum_{i=1}^{k-1} p^i}(A).$ Again, by the fact that the operation is strictly defined, the order $k$ type 1 Frobenius product is a well-defined element of 
$$
\frac{H^{p^n(|x|+|y|) - \sum_{i=1}^n p^i}(A)}{H^{p^{n-1}(|x|+|y|) - 2^{n-1}+ \sum_{i=0}^n p^i}(A)^p}.
$$
It follows that, if the class the order $k$ type 1 Frobenius product in the above quotient is 0, it is always possible to find $c_k^p = dc_{k+1}$. By the same argument as before, observe that  $c_k^p = d(c_{k+1}+\sigma + \tau_1 +\tau_2)$ for all cocycles $\sigma$ and $d\tau_1 = a^{p^{n-1}} P^p$ and   $d\tau_2 =b^{p^{n-1}} Q^p $. The cocycle $\sigma$ accounts for the $H^{p^{n-1}(|x|+|y|) - 2^{n-1}+ \sum_{i=0}^n p^i}(A)^p$  in the indeterminacy calculation. The $\tau_1$ and $\tau_2$ accounts $x^{p^n}H^{p^n(|y|) - \sum_{i=1}^n p^i}(A)+ y^{p^n}H^{p^n(|x|) - \sum_{i=1}^n p^i}(A)$ in the indeterminacy. This shows that the order $(i+1)^{th}$ operation has the correct indeterminacy. Moreover the order $k+1$ Frobenius product set has the correct form by the same reasoning as in the order 2 case.
\end{proof}
\section{Higher Steenrod operations as obstructions to rectifiablity}
The purpose of this subsection is to set up an obstruction theory for commutativity, paralleling the obstruction theory for formality given by Massey products. Our obstructions will be given by \emph{higher Steenrod products}. The first application of this theory is the following well-known folklore result that the author learned from lecture material by Mandell \cite{mandell09}.

\begin{proposition}
\label{prop:spaces_are_never_commutative}
    Let $X$ be a topological space. The $E_\infty$-algebra $ C^\ast(X, \mathbb F_p)$ admits a strictly commutative model only if $X$ is weakly homotopy equivalent to the disjoint union of contractible spaces. 
\end{proposition}
\begin{proof}
    Suppose towards contradiction that $E_\infty$-algebra $ C^\ast(X, \mathbb F_p)$ admitted a commutative model in the category of $\mathbb F_p$-commutative dg-algebras. Recall that $ C^\ast(X, \mathbb F_p)$ admits Steenrod operations on its cohomology. Such operations are preserved by quasi-isomorphisms of $E_\infty$-algebras. All of these operations vanish on strictly commutative dg-algebras except for $\operatorname{P}^nx$ when $|x|=n.$ In particular, the zeroth Steenrod operation $\operatorname{P}^0x$ is always $x$ on cohomology of $ C^\ast(X, \mathbb F_p)$, while $\operatorname{P}^0x$ vanishes on the cohomology of commutative dg-algebra, except when $|x|=0.$ It follows that  $ C^\ast(X, \mathbb F_p)$ admits a commutative model only if its cohomology is concentrated in degree 0.
\end{proof}
We now proceed to define \emph{higher Steenrod operations}. These will be our obstructions to commutativity. These are essentially the subset of cotriple products given by the primary Steenrod operations, and higher obstructions formed by syzgyies of Steenrod operations.

\begin{definition}
Consider the map of operads $\mathcal E \to \operatorname{Com}$. Let $(\E(\bigoplus^N_{i = 0} V_i), d) $ be an $N$-step Sullivan model. Then the \emph{Sullivan projection map} is the map of $\E$-algebras
$$
\pi_N: (\E(\bigoplus^N_{i = 0} V_i), d) \to (\Sym(\bigoplus^N_{i = 0} V_i), d).
$$
 and is defined by induction on $i$ as follows.  When $i = 0;$ $\pi_0: \E(V_0) \to \Sym(V_0)$ is the $\E$-algebra map directly induced by the map $\E\to \operatorname{Com}$. Therefore we assume that there is a map $\pi_k: (\E(\bigoplus^k_{i = 0} V_i), d) \to (\Sym(\bigoplus^k_{i = 0} V_i), d).$ We define $(\Sym(\bigoplus^k_{i = 0} V_i), d)$ on generators via the attachment map 
 $$dV_{k+1} \to \E(\bigoplus^k_{i = 0} V_i), d) \to (\Sym(\bigoplus^k_{i = 0} V_i), d)$$
 and extend this as a derivation. The map $\pi_k$ therefore also extends to 
 $$
 \pi_{k+1}: (\E(\bigoplus^{k+1}_{i = 0} V_i), d) \to (\Sym(\bigoplus^{k+1}_{i = 0} V_i), d)
 $$
 by sending $V_{k+1}$ to itself.
\end{definition}
\begin{definition}
    Let $\mathcal E$ be a model for the $E_\infty$-operad, $A$ be a $\mathcal E$-algebra and let $\sigma \in I( \mathcal E(\bigoplus_{i=1}^N V_i), d))$ be a cocycle. There is a unique quasi-isomorphism $\mathcal E \to \mathsf{Com}.$ Then $\sigma$ defines a  \textit{higher Steenrod operation of order $N$} if it appears in the kernel of the Sullivan projection map 
    $$
    \mathcal E(\bigoplus_{i=1}^N V_i), d) \to  \Sym(\bigoplus_{i=1}^N V_i), d).
    $$
    
\end{definition}
\begin{corollary}
    Let $A$ be an $\E$-algebra. Suppose $A$ admits a higher Steenrod operation that does not vanish as a differential in the cotriple spectral sequence. Then $A$ is not rectifiable.
\end{corollary}
\begin{proof}
    Any non-commutative higher Steenrod operation is always identically zero on a strictly commutative dg-algebra $B$. This is because it can be written in terms of a defining system in which every operation vanishes on $B$. Moreover  higher Massey operations are preserved by quasi-isomorphisms of $E_\infty$-algebras as differentials in the cotriple spectral sequence. Therefore $A$ cannot be quasi-isomorphic to a commutative dg-algebra.
\end{proof}
\subsection{Necessary and sufficient condition for rectifiability}

The purpose of this section is to show that our obstruction theory for commutativity is  complete. In other words, we shall give a necessary and sufficient condition for an arbitrary $E_\infty$-algebra $A$ over $\mathbb F_p$ to have a commutative model. Note that in this case, by Proposition \ref{prop:spaces_are_never_commutative} $A$ will never have the homotopy type of a space.  Our result is inspired by the following classical result which we state first.
\begin{theorem}
\label{deligne}
\cite{deligne75}
    Let $A$ be a commutative dg-algebra in $\mathbb Q$-vector spaces. Let $\mathfrak m = (\mathbb \Sym(\bigoplus_{i=0}^\infty V_i), d)$ be the minimal model for $A$. Then $A$ is formal if and only if, there is in each $V_i$ a complement $B_i$ to the cocycles $Z_i$, $V_i= Z_i \oplus B_i$, such that any closed form, $a$, in the ideal, $I((\bigoplus_{i=0}^\infty B_i)$, is exact. 
\end{theorem}
\begin{remark}
    The condition stated in this theorem is often referred as to the \emph{coherent vanishing} of Massey products. The reason for this is that any closed form $a$ in the ideal $I((\bigoplus_{i=0}^\infty B_i)$ is a Massey product in the sense of Definition \ref{def:product_set}, since the minimal resolution can be upgraded to a Sullivan resolution.  The condition that $a$ is exact is precisely the requirement that it vanish in cohomology.
\end{remark}
When we are not working in the rational setting, there is no longer a preferred choice of cofibrant resolution like the minimal model. Therefore our statement will be stated in the language of Sullivan resolutions.
\begin{definition}
\label{def:coherentvanishing}
    Let $A$ be an $E_\infty$-algebra over $\mathbb F_p$. Then the higher Steenrod operations \emph{vanish coherently} if for every Sullivan resolution $(\mathcal E(\bigoplus_{i=0}^\infty V_i), d)$ for $A$, there exists a splitting $V_i=X_i\bigoplus Y_i,$ with $X_0 = V_0$; such that $(\Sym(\bigoplus_{i=0}^{\infty}X_i), d)$ is a Sullivan algebra and the kernel of
    $$
(\mathcal E( \bigoplus_{i=0}^{\infty}V_i), d )\to  (\Sym(\bigoplus_{i=0}^{\infty}X_i), d)
$$
is acyclic.
\end{definition}
It is worth noting that if \textbf{any} choice of Sullivan resolution admits such a splitting, then the proof of  Theorem \ref{thm:rectification} implies that there is such a splitting for \textbf{every} choice. In this sense, the non-uniqueness of Sullivan resolutions does not change the theory significantly from the rational case.
\begin{remark}
The cocycles appearing in the kernel represent Steenrod operations. For example, the kernel of the $\E(V_0)\to \Sym(V_0)$ component are precisely the Steenrod operations and the definition of a Sullivan algebra immediately implies that these extra cocycles are killed by $Y_1.$ 
\end{remark}
While this definition is clearly highly reminiscent of coherent vanishing of Massey products, it is not precisely the same. Firstly, Theorem \ref{deligne} uses minimal models, while coherent vanishing of cotriple products uses Sullivan resolutions, which will generally be much larger. Secondly, minimal models are unique, while there are many choices of Sullivan resolution. Thirdly, coherent vanishing of Massey products involves the vanishing of all higher order operations, while coherent vanishing of cotriple operations does not require either Massey products or the higher commutative operations of section 4 to vanish.

\begin{theorem}
\label{thm:rectification}
    Let $A$ be an $E_\infty$-algebra over $\mathbb F_p$. Then $A$ is rectifiable if and only if its higher Steenrod operations vanish coherently.
\end{theorem}

\begin{proof}
First we prove the \emph{only if} direction. That is to say that we first suppose that $A$ is rectifiable and we then we shall show that every Sullivan model for $A$ admits a splitting such that the conditions of Definition \ref{def:coherentvanishing} are satisfied. If $A$ is rectifiable it has a strictly commutative model $\bar{A}$. Since Sullivan models are cofibrant, it follows from Proposition \ref{prop:sullivan_algebras_respect_weak_eq} that one has a map 
$$
f: (\E(\bigoplus^\infty_{i=0} V_i),d)\to \bar{A}
$$
satisfying the axioms of a Sullivan algebra.
We build the desired splitting by induction on $i$.
Firstly  let $X_0 = V_0$ and $Y_0 = 0$. Since $\bar{A}$ is strictly commutative, there is a factorisation
\[
\begin{tikzcd}
\E(V_0) \arrow[r, "g_0"] \arrow[rd, "f|_{\E(V_0)}"'] &  \Sym(X_0) \arrow[d, "h_0"] \\
& \bar{A}
\end{tikzcd}
\]
Consider the map 
$$
V_1 \xrightarrow{d} \E(V_0)\xrightarrow{g_0} \Sym(X_0)
$$
There is a splitting $V_1 = X_1\oplus Y_1$ such that $Y_1$ is the kernel of this map and $X_1$ is some complement to it.

Now inductively, we assume the splitting $V_i = X_i \oplus Y_i$ exists for $i \leq k$ and moreover that there is a factorisation
\[
\begin{tikzcd}
\E((\bigoplus_{i=0}^k V_i), d) \arrow[r, "g_k"] \arrow[rd, "f_k"'] &  \Sym((\bigoplus_{i=0}^k X_i), d)\arrow[d, "h_k"] \\
& \bar{A}
\end{tikzcd}
\]
Then there is a splitting $V_{k+1} = X_{k+1}\oplus Y_{k+1}$ such that $Y_{k+1}$ is the kernel of the map 
$$
V_{k+1} \xrightarrow{d} \E((\bigoplus_{i=0}^{k+1} V_i), d)\xrightarrow{g_k} \Sym((\bigoplus_{i=0}^k X_i)
$$
and $X_{k+1}$ is some complement to it. The existence of the factorisation once again follows from the fact $\bar{A}$ is commutative.

Lastly, we verify that this splitting satisfies the condition that the kernel of the projection
$$
(\mathcal E(\bigoplus_{i=0}^{\infty}V_i), d) \to  (\Sym(\bigoplus_{i=0}^{\infty}X_i), d)
$$
is acyclic. Suppose that $\sigma \in (\mathcal E(\bigoplus_{i=)}^{\infty}V_i), d)$ is a cocycle in the kernel. Then we have that 
$\sigma \in (\mathcal E(\bigoplus_{i=0}^{N}V_i), d)$ for some $N.$ Since the map $f_N$, by construction, factors through $\Sym(\bigoplus_{i=1}^{N}X_i), d)$ it follows that $f(\sigma)= 0.$ It then follows from the second condition of Definition \ref{def:sullivan_algebra}, that there is $\tau \in V_{k+1}$ such that $d\tau = \sigma.$ It is clear from our definition of the splitting that $\tau\in Y_{N+1}$.

\medskip

Conversely, suppose that $A$ is an $E_\infty$-algebra over $\mathbb F_p$ such that its higher Steenrod operations vanish coherently. Then, by definition, there is a quasi-isomorphism
$$
(\mathcal E(\bigoplus_{i=0}^\infty X_i\oplus Y_i), d) \xrightarrow{\sim} A.
$$
where $(\mathcal E(\bigoplus_{i=0}^\infty X_i\oplus Y_i), d)$ satisfies the hypotheses of Definition \ref{def:coherentvanishing}. We claim that the projection map
$$
f:(\mathcal E(\bigoplus_{i=0}^\infty X_i\oplus Y_i), d) \to \operatorname{Sym}(\bigoplus_{i=0}^\infty X_i)
$$
is a quasi-isomorphism. The map $f$ is surjective so, by the long exact sequence in cohomology, it suffice to prove that the kernel of $f$ is acyclic. This is precisely the coherent vanishing condition.
\end{proof}
The previous result has the following corollary; which is proven similarly.
\begin{definition}
    Let $A$ be an $\E$-algebra. We say that it is \emph{formal} if it is quasi-isomorphic to $H^\ast(A)$, regarding the cohomology as a commutative algebra.
\end{definition}
It is clear that formality implies rectifiability. However the opposite is not true. In particular, Massey products are examples of cotriple products that are not, with rare exception, obstructions to rectifiability, but are an obstruction to formality.
\begin{definition}
    Let $A$ be an $E_\infty$-algebra over $\mathbb F_p$. Then the cotriple products \emph{vanish coherently} if for every Sullivan resolution $(\mathcal E(\bigoplus_{i=0}^\infty V_i), d)$ for $A$, the ideal $I(dV_1\oplus \mathcal E(\bigoplus_{i=1}^\infty V_i), d)$ is acyclic.
\end{definition}
\begin{corollary}
Let $A$ be an $E_\infty$-algebra over $\mathbb F_p$. Then $A$ is formal as an $\mathcal E$-algebra if and only if its cotriple operations all vanish coherently.   
\end{corollary}
\begin{proof}
    First we prove the \emph{only if} direction. That is to say that we first suppose that $A$ is formal and we then we shall show that every Sullivan model for $A$, the cotriple operations vanish coherently. It follows from Proposition \ref{prop:sullivan_algebras_respect_weak_eq} that one has a map 
$$
f: (\E(\bigoplus^\infty_{i=0} V_i),d)\to H^\ast(A)
$$
that makes $\left((\E(\bigoplus^\infty_{i=0} V_i),d),f\right)$ a Sullivan model for $H^*(A)$. This map is clearly surjective and therefore the kernel of $f$ must be acyclic. If we show that the kernel is isomorphic to $I(dV_1\oplus \mathcal E(\bigoplus_{i=1}^\infty V_i), d)$, we can conclude the result by the long exact sequence in cohomology. First note that the map $f|_{V_0} = \id_H$. Let 
$$W_i = \{v-\left(f|_{V_0}\right)^{-1}\left(f(v)\right): v\in V_i\}$$
Then, one can easily verify that the kernel of $f$ is equal to  $I(dV_1\oplus \mathcal E(\bigoplus_{i=1}^\infty W_i), d)$ and this is isomorphic to the ideal $I(dV_1\oplus \mathcal E(\bigoplus_{i=1}^\infty V_i), d)$. In particular, both have the same cohomology .

\medskip

Conversely, suppose that $A$ is an $E_\infty$-algebra over $\mathbb F_p$ such that its cotriple operations vanish coherently. Then, by definition, there is a quasi-isomorphism
$$
(\mathcal E(\bigoplus_{i=0}^\infty V_i), d) \xrightarrow{\sim} A.
$$
where $I(dV_1\oplus \mathcal E(\bigoplus_{i=1}^\infty V_i), d)$ is acyclic. But this is the kernel of the algebra map
$$
(\mathcal E(\bigoplus_{i=0}^\infty V_i), d) \to V_0 = H^\ast(A)
$$
so, by the long exact sequence in cohomology, we conclude that this map is an isomorphism, and so $A$ is formal.
\end{proof}

\bibliographystyle{plain}
\bibliography{MyBib}

\bigskip

\noindent\sc{Oisín Flynn-Connolly}\\ 
\noindent\sc{ Leiden University,  Leiden, The Netherlands}\\
\noindent\tt{oisinflynnconnolly@gmail.com}\\

\noindent The author completed the majority of this work while holding an affiliation with:

\qquad

\noindent\sc{Université Sorbonne Paris Nord \\ Laboratoire de Géométrie, Analyse et Applications, LAGA \\ CNRS, UMR 7539, F-93430  Villetaneuse, France}\\

\end{document}